\documentclass[reqno, 11pt]{amsart}

\usepackage[text={155mm,235mm},centering]{geometry}
\input diagxy
\input xy

\usepackage[active]{srcltx} 

\newtheorem{thm}{Theorem}[section]
\newtheorem{main thm}{Main Theorem}[section]
\newtheorem{cor}[thm]{Corollary}
\newtheorem{lem}[thm]{Lemma}

\theoremstyle{definition}
\newtheorem{defn}[thm]{Definition}
\theoremstyle{property}

\theoremstyle{remark}
\newtheorem{rem}[thm]{Remark}

\numberwithin{equation}{section}
\usepackage[mathscr]{eucal}
\usepackage[all]{xy}
\usepackage{mathrsfs}
\usepackage{xypic}
\usepackage{amsfonts}
\usepackage{amsmath}
\usepackage{amsthm,bm}
\usepackage{amssymb,tikz-cd}
\usepackage{latexsym}
\usepackage{tabularx}
\usepackage{graphicx}
\usepackage{pict2e}
\usepackage[pagebackref]{hyperref}
\usepackage{tikz}
\usepackage{textcomp}
\usepackage[scr=rsfs]{mathalfa}

\definecolor{ceruleanblue}{rgb}{0.16, 0.32, 0.75}
\hypersetup{colorlinks=true,allcolors=ceruleanblue}
\allowdisplaybreaks

\begin{document}

\title[The Hodge and Grothendieck's standard conjectures  for fiber bundles]{Algebraic cycles on fiber bundles of Leray-Hirsch type}

\author{Lingxu Meng}
\address{Department of Mathematics, North University of China, Taiyuan, Shanxi 030051,  P. R. China}
\email{menglingxu@nuc.edu.cn}%

\subjclass[2010]{Primary 14C25, 14C15, 14C30}
\keywords{Hodge conjecture; standard conjectures;  Leray-Hirsch theorem; Chow group; fiber bundle;  product variety}


\begin{abstract}
  We give a theorem of Leray-Hirsch type for Chow groups and use it to study the Hogde  and Grothendieck's standard conjectures for algebraic fiber bundles of Leray-Hirsch type.  Morevoer, the Hodge conjecture for  product varieties will also be considered.
\end{abstract}

\maketitle


\section{Introduction}
Algebraic cycle is a central  object in algebraic geometry, which plays an important role in the study of the geometry, topology  and arithmetic of algebraic varieties.
There are many important conjectures on algebraic cycles.

W. Hodge \cite{Ho}  posed a question: On a smooth complex projective variety, is any integral Hodge class   algebraic?
 M.  Atiyah  and F. Hirzebruch  \cite{AH} first gave it an counterexample and pointed that there are extra conditions
for the torsion class to be algebraic.
A. Grothendieck \cite{G} observed that it is necessary to use  rational coefficients  in place of integral ones.
Now, the corrected version is  the well-known  \emph{Hodge conjecture}  and Hodge's original one is called  the \emph{integral Hodge conjecture}.
For a smooth projective variety $X$, denote by $\textrm{Hodge}(X,\mathbb{Z})$ the statement: \emph{The integral Hodge conjecture is true for $X$} and  by $\textrm{Hodge}(X,\mathbb{Q})$ the statement: \emph{The Hodge conjecture is true for $X$}. Evidently, $\textrm{Hodge}(X,\mathbb{Z})$ implies $\textrm{Hodge}(X,\mathbb{Q})$.

For understanding the Weil conjectures, A. Grothendieck \cite{G0} posed several conjectures on algebraic cycles for smooth projective varieties over algebraically closed ground fields. Here, we only describe them over $\mathbb{C}$.
Suppose that $X$ is a smooth complex projective variety with dimension $n$ and set $h^p(X)=\textrm{dim}_{\mathbb{Q}}H_{alg}^{2p}(X,\mathbb{Q})$.

$\bullet$ the conjecture $A(X)$: $h^p(X)=h^{n-p}(X)$ for all $p\in \mathbb{N}$;

$\bullet$ the  conjecture of Lefschetz type $B(X)$:  the correspondence $\Lambda\in H^{2n-2}(X,\mathbb{Q})$ associated to the Lefschetz operatior $L$ is algebraic;

$\bullet$ the K\"{u}nneth conjecture $C(X)$: all components of the class $[\Delta]^{top}\in H^{2n}(X\times X,\mathbb{Q})$ of diagonal under the K\"{u}nneth decomposition are algebraic;

$\bullet$ the conjecture $D(X)$: the numerical and homological equivalences of algebraic cycles on $X$ coincide.

These conjectures are not independent of each other. For relations among them, we refer to \cite{G0,K,K2,L,M}.
The Hodge conjecture and the standard conjectures are so far away from being solved.
In this paper, we will concentrate on the problem: when the algebraic fiber bundle satisfies these conjectures.
There  is not many results in this aspect, see \cite{Z,V}.
We study  these conjectures on the algebraic fiber bundles of Leray-Hirsch type (see Definition \ref{def}) and the main result is
\begin{thm}\label{L-H2}
Assume that the complex projective variety $E$ is an algebraic fiber bundle  over a smooth complex projective variety $X$.

$(1)$ Suppose that $E$ is of Leray-Hirsch type. Then $\emph{Hodge}(E,\mathbb{Z})$  if and only if $\emph{Hodge}(X,\mathbb{Z})$.

$(2)$ Suppose that $E$ is of $\mathbb{Q}$-Leray-Hirsch type. Then $S(E)$ if and only if $S(X)$, where $S(\bullet)$ means one of $A(\bullet)$, $B(\bullet)$, $C(\bullet)$, $D(\bullet)$, $\emph{Hodge}(\bullet,\mathbb{Q})$.
\end{thm}

Product varieties can be viewed as trivial fiber bundles.
An interesting problem is whether the integral Hodge conjecture or Hodge conjecture is  true for  product varieties  if it  is true for all factors.
A.  Conte and  J. Murre \cite{CM} proved that the Hodge conjecture holds on the product of a smooth threefold and  a smooth rational curve.
F. Hazama \cite{Ha} investigated the algebraic cycles on nonsimple Abelian varieties and then gave an affirmative answer of the Hodge conjecture  for the product of two stably nondegenerate Abelian  varieties which have no simple components of type IV.
J. Ram\'{o}n Mar\'{\i} \cite{RM} and   T. Shioda \cite{Sh} studied the Hodge conjecture for  products of  surfaces and  products of Fermat varieties respectively.
Z. Xu \cite{X} proved the Hodge conjecture holds for arbitrary products of generalized Kummer varieties.
We will consider the following cases  from the viewpoint of algebraic fiber bundles.
\begin{thm}\label{product}
Let $X$ and $Y$  be smooth complex projective varieties.

$(1)$ If $H^*(Y,\mathbb{Z})$ is torsion free and algebraic, then $\emph{Hodge}(X\times Y,\mathbb{Z})$  if and only if $\emph{Hodge}(X,\mathbb{Z})$.

$(2)$ If $H^{p,q}(Y)=0$ for all $p$, $q$ satisfying $p\neq q$ \emph{(}e.g., if $H^*(Y,\mathbb{Q})$ is algebraic\emph{)}, then $\emph{Hodge}(X\times Y,\mathbb{Q})$ if and only if $\emph{Hodge}(X,\mathbb{Q})$ and $\emph{Hodge}(Y,\mathbb{Q})$.
\end{thm}

For example, if the Hodge conjecture is true for $X$, then it is true for $S_1\times S_2\times \ldots \times S_n \times X$, where $S_i$ is a smooth complex  projective surface with $p_g(S_i)=q(S_i)=0$ (\emph{e.g.}, rational, Enriques, del Pezzo,  Godeaux, Campedelli, Burniat,  Hirzebruch surfaces, \emph{etc.}) for $1\leq i\leq n$.
Moreover, any smooth complex  projective cellular variety $Y$  satisfies the assumptions in Theorem \ref{product}, since $H^*(Y,\mathbb{Z})$ is a free abelain group with a basis consisting of algebraic classes (\cite[Example 1.9.1, 19.1.11]{F}).

\subsection*{Acknowledgements}
The author is supported by Scientific and Technological Innovation Programs of Higher Education Institutions in Shanxi (Grant No. 2020L0290), the National Natural Science Foundation of China (Grant No. 12001500, 12071444) and the Natural Science Foundation of Shanxi Province of China (Grant No. 201901D111141).


\section{Preliminaries}
\subsection{Notations}
Let $k$ be a field.
A \emph{variety} is a separated scheme of finite type over the field $k$.
Denote by $X(k)$ the set of closed points of $X$.

Recall some notations, see \cite[Chapter 17]{F} for detail definitions.
For any  variety $X$, denote its Chow group by $\textrm{CH}_*(X)$ and Chow cohomology  by $A^*(X)$ (\cite[Definition 17.3]{F}).
In general, $\textrm{CH}_*(X)$ may be not a ring for a singular variety.
For a subvariety $V\subseteq X$, write $[V]$ as the class of $V$ in $\textrm{CH}_*(X)$.
There exists a cup product  on $A^*(X)$, denoted by $\cdot$,  which make $A^*(X)$  a  graded commutative ring with unit $\mathbf{1}_X$ (\cite[Example 17.4.4 (a)]{F}).
There exists a cap product $\cap:A^p(X)\otimes \textrm{CH}_q(X)\rightarrow \textrm{CH}_{q-p}(X)$ satisfying that $\mathbf{1}_X\cap \alpha =\alpha$.
For a morphism $f:X\rightarrow Y$ of varieties, there exists a ring homomorphism $f^*:A^{*}(Y)\rightarrow A^*(X)$ (\cite[p. 324]{F}).
Set $n=\textrm{dim}_kX-\textrm{dim}_kY$.
If $f$ is a flat or local complete intersection morphism (\cite[B. 8.1.2]{F}), there exists a Gysin homomorphism
$f^*:\textrm{CH}_*(Y)\rightarrow \textrm{CH}_{*+n}(X)$  (\cite[p. 328 ($G_1$)]{F}).
Moreover, if $f$ is proper, there exists another Gysin homomorphism $f_*:A^*(X)\rightarrow A^{*-n}(Y)$ (\cite[p. 328 ($G_2$)]{F}).
In particular, if $i:X\rightarrow Y$ is a closed embedding, set $\alpha|_X=i^*\alpha$ for $\alpha\in A^*(Y)$ or $\textrm{CH}_*(Y)$.
In the category of smooth varieties,

(1)  (\cite[Corollary 17.4]{F}) $\bullet\cap[X]:A^{*}(X)\tilde{\rightarrow} \textrm{CH}_{\textrm{dim}_kX-*}(X)$ is an isomorphism
and $(\alpha\cap [X])\cdot (\beta\cap [X])=(\alpha\cdot \beta)\cap [X]$ for $\alpha$, $\beta\in A^*(X)$;

(2)  (\cite[Example 17.4.3 (c)]{F}) the Gysin homomorphism $f^*$ coincides with the usual pullback of the Chow groups (\cite[p. 113]{F}).

On a complete variety $X$, the degree of $\alpha\in \textrm{CH}_0(X)$ is denoted by $\int_X\alpha$ (\cite[Defintion 1.4]{F}).
Let $f:X\rightarrow Y$ be a proper smooth surjective morphism  of varieties.
For a general  point $y\in Y(k)$, the fiber $X_y=f^{-1}(y)$ is a complete variety with dimension $n=\textrm{dim}_k X-\textrm{dim}_k Y$.
Denote by $i:\{y\}\rightarrow Y$  the  inclusion.
For any $\alpha\in A^n(X)$,
\begin{displaymath}
i^*f_*(\alpha)=(f|_{X_y})_*(\alpha|_{X_y})=\left(\int_{X_y}\alpha|_{X_y}\cap [X_y]\right)\cdot \mathbf{1}_{\{y\}}
\end{displaymath}
by \cite[Example 17.4.1 (b)]{F}.
Since $A^0(Y)=\mathbb{Z}\cdot \mathbf{1}_Y$, $i^*:A^0(Y)\rightarrow A^0(\{y\})$ is an isomorphism, so $f_*(\alpha)=\left(\int_{X_y}\alpha|_{X_y}\cap [X_y]\right)\cdot \mathbf{1}_Y$.

An \emph{algebraic fiber bundle} with general fiber $F$ means a morphism $\pi:E\rightarrow X$ of (not necessarily smooth or projective) varieties satisfying that there is an (Zariski) open covering $\mathfrak{U}$ of $X$  such that  $\pi:\pi^{-1}(U)\rightarrow U$ is isomorphic to the projection from $U\times F$ to $U$ for every $U\in \mathfrak{U}$.

\subsection{The complex cases}
Let $X$ be a smooth complex projective variety.
Denote by  $H^*(X,\mathbb{Z})$ (resp. $H^*(X,\mathbb{Q})$) the Betti cohomology of $X$, i.e., the singular cohomology of $X(\mathbb{C})$ under the analytic topology with integral (resp.  rational) coefficients.
Set $Hdg^{p}(X,\mathbb{Z})=H^{2p}(X,\mathbb{Z})\cap H^{p,p}(X)\subseteq H^{2p}(X,\mathbb{Z})$ and  $Hdg^{*}(X,\mathbb{Z})=\bigoplus_{p\geq 0} Hdg^p(X,\mathbb{Z})$, where $H^{p,p}(X)$ is the $(p,p)$-th Dolbeault cohomology of $X(\mathbb{C})$.
Using $\mathbb{Q}$ instead of $\mathbb{Z}$, we can define  $Hdg^{*}(X,\mathbb{Q})$ similarly.
There is a natural isomorphism $Hdg^{*}(X,\mathbb{Z})\otimes_{\mathbb{Z}}\mathbb{Q}\cong Hdg^{*}(X,\mathbb{Q})$.

Denote by $\textrm{CH}^*(X)$ (=$\textrm{CH}_{\textrm{dim}_{\mathbb{C}}X-*}(X)$)  the  Chow ring of $X$.
Denote the \emph{cycle map} by $\textrm{cl}^*_X:\textrm{CH}^*(X)\rightarrow H^*(X,\mathbb{Z})$ and the \emph{refined cycle map} $\lambda^*_X:\textrm{CH}^*(X)\rightarrow Hdg^*(X,\mathbb{Z})$, which  coincides with each other except ranges.
Define $\textrm{CH}^*(X)_{\mathbb{Q}}=\textrm{CH}^*(X)\otimes_{\mathbb{Z}}\mathbb{Q}$,
$\textrm{cl}^*_{X,\mathbb{Q}}=\textrm{cl}^*_X\otimes 1: \textrm{CH}^*(X)_{\mathbb{Q}}\rightarrow H^*(X,\mathbb{Q})$ and $\lambda^*_{X,\mathbb{Q}}=\lambda^*_X\otimes 1: \textrm{CH}^*(X)_{\mathbb{Q}}\rightarrow Hdg^*(X,\mathbb{Q})$.
Let $V$ be a subvariety of $X$.
To avoid confusions, $[V]^{top}$ and $[V]^{alg}$ means the class of $V$ in $H^*(X,\mathbb{Z})$ and $\textrm{CH}^*(X)$ respectively.
Clearly, $\textrm{cl}^*_X([V]^{alg})=[V]^{top}$.

Set $H_{alg}^*(X,\mathbb{Z})=\textrm{Im} \textrm{cl}^*_{X}$ (resp. $H_{alg}^*(X,\mathbb{Q})=\textrm{Im} \textrm{cl}^*_{X,\mathbb{Q}}$).
Obviously, $H_{alg}^*(X,\mathbb{Z})\subseteq Hdg^{*}(X,\mathbb{Z})$ (resp. $H_{alg}^*(X,\mathbb{Q})\subseteq Hdg^{*}(X,\mathbb{Q})$).
A class in $H_{alg}^*(X,\mathbb{Z})$ (resp. $H_{alg}^*(X,\mathbb{Q})$) is said to be an \emph{algebraic class} and a subset $A\subseteq H_{alg}^*(X,\mathbb{Z})$ (resp. $H_{alg}^*(X,\mathbb{Q})$) is said to be \emph{algebraic}.
The integral Hodge conjecture $\textrm{Hodge}(X,\mathbb{Z})$ (resp. Hodge conjecture  $\textrm{Hodge}(X,\mathbb{Q})$) means that  $Hdg^{*}(X,\mathbb{Z})$  (resp. $Hdg^{*}(X,\mathbb{Q})$) is algebraic, i.e.,  $H_{alg}^*(X,\mathbb{Z})= Hdg^{*}(X,\mathbb{Z})$ (resp. $H_{alg}^*(X,\mathbb{Q})= Hdg^{*}(X,\mathbb{Q})$), or equivalently, $\lambda^*_{X}$ (resp. $\lambda^*_{X,\mathbb{Q}}$) is surjective.

\section{Leray-Hirsch theorem for Chow groups}
In this subsection, assume that $k$ is a field and any variety is over $k$.

\subsection{The suitable conditions}
For proving Theorem \ref{L-H2}, we need a theorem  of  Leray-Hirsch type for Chow groups.
Unlike the cases in topology,  the Leray-Hirsch theorem for Chow groups is not true in general cases.
It is necessary  to add some assumptions on  fibers. We only consider the case that the general fiber is  smooth and complete.
A. Collino and W. Fulton \cite[Appendix C]{CF}  gave a version of Leray-Hirsch theorem under the assumptions that the general fiber is a cellular  variety and satisfies the Poincar\'{e} duality property.
In a different way, we will provide a version of Leray-Hirsch theorem, which is suitable for the most general cases in some sense.

For any algebraic fiber bundle $E$ over $X$, the condition $\textbf{(L-H)}$ of Leray-Hirsch type means:

\emph{There exist cycle classes $e_1$, $\ldots$, $e_r$  in $A^*(E)$  such that their restrictions $e_1|_{E_x}$, $\ldots$, $e_r|_{E_x}$ freely generate $A^*(E_x)$  for every point $x\in X(k)$}.\\
Clearly, if a bundle with general fiber $F$ satisfies  $\textbf{(L-H)}$,  $A^*(F)$ is a free abelian group.
By analogy with the situation in topology, the Leray-Hirsch theorem  should be stated  as the following form (where $\textbf{(C)}$ is a undetermined condition):

\emph{Let  $\pi:E\rightarrow X$ be an algebraic fiber bundle with the general fiber $F$.
Assume that $F$ is a smooth complete variety satisfying the condition $\textbf{\emph{(C)}}$.
Suppose that  $E$ satisfies $\textbf{\emph{(L-H)}}$.
Then
\begin{displaymath}
\sum_{i=1}^{r}e_i\cap\pi^*(\bullet):\bigoplus_{i=1}^{r}\emph{CH}_{*+u_i-n}(X)\rightarrow \emph{CH}_{*}(E)
\end{displaymath}
is an isomorphism, where $e_i\in A^{u_i}(E)$ for $1\leq i\leq r$.}

Now, we attempt to find the  assumption $\textbf{(C)}$ as weak as possible for the such version of Leray-Hirsch theorem.
Let $X$ be  any variety and $F$ a smooth complete variety.  Suppose that $F$  satisfies the condition $\textbf{(C)}$ and $A^*(F)$ is a free abelian group.
Let $pr_1$ and $pr_2$ be the two projections from $X\times F$ onto $X$ and $F$ respectively.
Suppose that  $f_1$, $\cdots$, $f_r$ is a homogeneous basis of $A^*(F)$, where $f_i\in A^{u_i}(F)$ for $1\leq i\leq r$.
Then $pr_2^*f_1$, $\cdots$, $pr_2^*f_r$ satisfies $\textbf{(L-H)}$.
Hence the trivial fiber bundle $pr_1:X\times F\rightarrow X$ satisfies the assumptions of Leray-Hirsch theroem.
Then
$\sum_{i=1}^{r}(pr_2^*f_i)\cap pr_1^*(\bullet):\bigoplus_{i=1}^{r}\textrm{CH}_{*+u_i-n}(X)\rightarrow \textrm{CH}_{*}(X\times F)$
should be an isomorphism.
By \cite[Corollary 17.4 (b)]{F}, $\sum_{i=1}^{r}(pr_2^*f_i)\cap pr_1^*(\bullet)=\sum_{i=1}^{r} \bullet\times f_i$.
So
\begin{equation}\label{weakest}
\times:\textrm{CH}_{*}(X)\otimes_{\mathbb{Z}}\textrm{CH}_*(F)\rightarrow \textrm{CH}_{*}(X\times F)
\end{equation}
should be an isomorphism.
Hence, \emph{a necessary condition that the Leray-Hirsch theorem holds is that \emph{(\ref{weakest})} is an isomorphism  for any variety $X$.}
In fact, this condition and $\textbf{(L-H)}$ are enough for the Leray-Hirsch theorem  for Chow groups, see the following Theorem \ref{L-H} and Lemma \ref{ppp}.
With the same argument, it is easy to obtain a similar one that the Leray-Hirsch theorem with rational coefficients holds   if we use $A^*(\bullet)$ instead of $A^*(\bullet)_{\mathbb{Q}}$ in $\textbf{(L-H)}$.

\subsection{Leray-Hirsch theorem}
A smooth complete variety $X$ is said to satisfy the \emph{K\"{u}nneth property} (resp. \emph{$\mathbb{Q}$-K\"{u}nneth property}), if  the exterior product
\begin{displaymath}
\times:\textrm{CH}_*(X)\otimes_{\mathbb{Z}}\textrm{CH}_*(Y)\rightarrow \textrm{CH}_*(X\times Y)
\end{displaymath}
(resp.
\begin{displaymath}
\times:\textrm{CH}_*(X)_{\mathbb{Q}}\otimes_{\mathbb{Q}}\textrm{CH}_*(Y)_{\mathbb{Q}}\rightarrow \textrm{CH}_*(X\times Y)_{\mathbb{Q}})
\end{displaymath}
is surjective for any variety $Y$.
All smooth complete linear schemes (in the sense of Totaro \cite[Section 3]{To}, e.g. cellular varieties) satisfy the K\"{u}nneth property by  \cite[Proposition 1]{To}.
Let $B$ be a Borel subgroup of a reductive group $G$ over the ground field $\mathbb{C}$, then the quotient $G/B$ is a cellular varieties (see \cite[p. 193. Theorem (b), Remark (3)]{B}).
Projective spaces, Grassmann varieties and  flag varieties are all examples of this type.

\begin{rem}
In the  usual  sense, the notation ``K\"{u}nneth property" should mean that  the exterior product is an isomorphism and our definition here looks weaker than that.
In fact, they are the same thing, see Corollary \ref{ppp}.
\end{rem}

\begin{lem}\label{point-open}
Let  $\pi:E\rightarrow X$ be an algebraic fiber bundle with general fiber $F$.
Assume that $F$ is a smooth complete variety and satisfies the K\"{u}nneth property \emph{(}resp. $\mathbb{Q}$-K\"{u}nneth property\emph{)}.
If $\alpha\in \emph{CH}_*(E)$ \emph{(}resp. $\alpha\in \emph{CH}_*(E)_{\mathbb{Q}}$\emph{)} satisfies that $\alpha|_{E_{x_0}}\in \emph{CH}_*(E_{x_0})$ \emph{(}resp. $\alpha|_{E_{x_0}}\in \emph{CH}_*(E_{x_0})_{\mathbb{Q}}$\emph{)} for a point $x_0\in X(k)$.
Then  $\alpha|_{E_U}=0\in \emph{CH}_*(E_U)$ \emph{(}resp. $\alpha|_{E_U}=0\in \emph{CH}_*(E_U)_{\mathbb{Q}}$\emph{)} for some nonempty open subset $U\subseteq X$, where $E_U=\pi^{-1}(U)$.
\end{lem}
\begin{proof}
We only prove the integral case.
Since $F$ satisfies  the K\"{u}nneth property, the class of the diagonal $F\times F$ is in the image of $\times:\textrm{CH}_*(F)\otimes_{\mathbb{Z}}\textrm{CH}_*(F)\rightarrow \textrm{CH}_*(F\times F)$.
By \cite[Theorem 2.1 (ii)]{ES2}, $\textrm{CH}_*(F)$ is a free abelian group.
Assume that  $\{\,\gamma_i\,|\,1\leq i\leq r\,\}$ is a homogeneous  $\mathbb{Z}$-basis of  $\textrm{CH}_*(F)$.
Let $V\subseteq X$ be an irreducible open neighborhood  of $x_0$ such that $E_V\cong V\times F$.
Since $F$ satisfy the K\"{u}nneth property, we have $\alpha|_{E_V}=\sum_{i=1}^r \beta_{i}\times \gamma_i$ for some $\beta_i\in \textrm{CH}_{p_i}(V)$ ($1\leq i\leq r$).
Set $n=\textrm{dim}_kX$.
For $p_i=n$,  there exists $n_i\in \mathbb{Z}$ such that $\beta_i=n_i[V]$.
Then
\begin{displaymath}
0=\alpha|_{E_{x_0}}=\sum_{i=1}^r \beta_{i}|_{\{x_0\}}\times \gamma_i=\sum\limits_{\substack{1\leq i\leq r\\p_i=n}} \beta_{i}|_{\{x_0\}}\times \gamma_i=\sum\limits_{\substack{1\leq i\leq r\\p_i=n}} n_{i}\cdot [x_0]\times \gamma_i,
\end{displaymath}
which implies that $\beta_{i}=0$ if $p_i=n$.
Set $Z=\bigcup\limits_{\substack{1\leq i\leq r\\p_i<n}}\textrm{supp} \beta_i$ and $U=V-Z$.
Clearly, $Z$ is a subvariety of $V$ with dimension $<n$ and $U$ is a nonempty open subset in $X$.
Moreover, if $p_i<n$, $\beta_i|_U=0$ in $\textrm{CH}_*(U)$.
So
\begin{displaymath}
\alpha|_{E_V}=\sum\limits_{\substack{1\leq i\leq r\\p_i<n}} \beta_{i}\times \gamma_i
\end{displaymath}
restricts to zero in $\textrm{CH}_*(U)\otimes_{\mathbb{Z}}\textrm{CH}_*(F)\cong\textrm{CH}_*(U\times F)=\textrm{CH}_*(E_{U})$.

\end{proof}

\begin{thm}\label{L-H}
Let  $\pi:E\rightarrow X$ be an algebraic fiber bundle with the general fiber $F$.
Assume that $F$ is a smooth complete variety satisfying the  K\"{u}nneth property \emph{(}resp. $\mathbb{Q}$-K\"{u}nneth property\emph{)}.
Suppose that  there exist cycle classes $e_1$, $\ldots$, $e_r$ of pure degrees $u_1$ , $\ldots$, $u_r$ respectively in $A^*(E)$  \emph{(}resp. $A^*(E)_{\mathbb{Q}}$\emph{)} such that their restrictions $e_1|_{E_x}$, $\ldots$, $e_r|_{E_x}$ freely generate $A^*(E_x)$ \emph{(}resp. $A^*(E_x)_{\mathbb{Q}}$\emph{)}  for every point $x\in X(k)$.
Then
\begin{displaymath}
\sum_{i=1}^{r}e_i\cap\pi^*(\bullet):\bigoplus_{i=1}^{r}\emph{CH}_{*+u_i-n}(X)\rightarrow \emph{CH}_{*}(E)
\end{displaymath}
\emph{(}resp.
\begin{displaymath}
\sum_{i=1}^{r}e_i\cap\pi^*(\bullet):\bigoplus_{i=1}^{r}\emph{CH}_{*+u_i-n}(X)_{\mathbb{Q}}\rightarrow \emph{CH}_{*}(E)_{\mathbb{Q}})
\end{displaymath}
is an isomorphism, where $n=\emph{dim}_kE-\emph{dim}_kX$.
\end{thm}
\begin{proof}
We only give a proof for the integral case and the rational case can be proved similarly.
Set $E_C=\pi^{-1}(C)$ for an open or a closed subset $C$ of $X$.
Let $i_C:C\rightarrow X$ and $i_{E_C}:E_C\rightarrow E$ be the inclusions.
Set
\begin{displaymath}
\Phi_C=\sum_{i=1}^{r}(i_{E_C}^*e_i)\cap(\pi|_{E_C})^*(\bullet):\bigoplus_{i=1}^{r}\textrm{CH}_{*+u_i-n}(C)\rightarrow \textrm{CH}_{*}(E_C).
\end{displaymath}
Our aim is to prove that $\Phi_X$ is an isomorphism.
Fix a point $x_0\in X(k)$.

Denote by $\Delta\in \textrm{CH}_*(E_{x_0}\times E_{x_0})$ the cycle class of the diagonal of $E_x$.
Since $E_{x_0}\cong F$ satisfies the K\"{u}nneth property,
there exist $\alpha_i\in A^*(E_{x_0})$ for $1\leq i\leq r$ such that $\Delta=\sum_{i=1}^r(\alpha_i\cap [E_{x_0}])\times (e_i|_{E_{x_0}}\cap [E_{x_0}])$.
Let $pr_1$ and $pr_2$ be the two projections of $E_{x_0}\times E_{x_0}$ onto $E_{x_0}$ respectively.
Then
\begin{displaymath}
\begin{aligned}
e_i|_{E_{x_0}}\cap [E_{x_0}]=&pr_{2*}\left(\Delta\cdot pr_1^*(e_i|_{E_{x_0}}\cap [E_{x_0}])\right)\\
= &\,\sum_{j=1}^rpr_{2*}\left(\left((e_i|_{E_{x_0}}\cap [E_{x_0}])\cdot (\alpha_j\cap [E_{x_0}])\right)\times (e_j|_{x_0}\cap [E_{x_0}])\right)\\
=&\,\sum_{j=1}^r\left(\int_{E_{x_0}}(e_i|_{E_{x_0}}\cdot \alpha_j)\cap [E_{x_0}]\right)\cdot (e_j|_{E_{x_0}}\cap [E_{x_0}]).
\end{aligned}
\end{displaymath}
Since $\{\,e_i|_{E_{x_0}}\,|\,i=1,\cdots,r\,\}$ is a $\mathbb{Z}$-basis of $A^*(E_{x_0})$,  $\int_{E_{x_0}}(e_i|_{E_{x_0}}\cdot \alpha_j)\cap [E_{x_0}]=\delta_{ij}$, where $\delta_{ij}$ is Kronecker's delta.
There exist $f_i\in A^*(E)$ such that $f_i|_{E_x}=\alpha_i$ for $1\leq i\leq r$, since the restriction $A^*(E)\rightarrow A^*(E_{x_0})$ is surjective.
Then
\begin{displaymath}
\pi_*(e_i\cdot f_j)=\left(\int_{E_{x_0}}(e_i\cdot f_j)|_{E_{x_0}}\cap [E_{x_0}]\right)\cdot\mathbf{1}_X=\delta_{ij}\cdot\mathbf{1}_X,
\end{displaymath}
where $\mathbf{1}_X\in A^0(X)$ is the unit of $A^*(X)$.
Assume that $\Phi_X(\beta_1,\ldots, \beta_r)=0$, i.e., $\sum_{i=1}^{r}e_i\cap\pi^*(\beta_i)=0$ in $\textrm{CH}_{*}(E)$.
By \cite[p. 26, ($\textrm{G}_4$)(ii)]{FM},
\begin{displaymath}
0=\pi_*\left(f_j\cap\left(\sum_{i=1}^{r}e_i\cap\pi^*(\beta_i)\right)\right)=\sum_{i=1}^{r}\pi_*(e_i\cdot f_j)\cap\beta_i=\beta_j.
\end{displaymath}
So $\Phi_X$ is injective.

We prove that $\Phi_X$ is surjective by induction for the dimension of $X$.
Clearly, $\Phi_X$ is surjective for $\textrm{dim}_k X=0$.
For $\textrm{dim}_k X>0$, we may assume that $X$ is irreducible.
For any $\alpha\in \textrm{CH}_{*}(E)$, $\alpha|_{E_{x_0}}=\gamma\cap [E_{x_0}]$ for some $\gamma\in A^{n-*}(E_{x_0})$, since $E_{x_0}$ is a smooth variety.
By the assumption of the theorem, $\gamma=\sum_{i=1}^rn_ie_i|_{E_{x_0}}$ for some integers $n_1$, \ldots, $n_r$.
Set
\begin{displaymath}
\beta=(n_1[X],...,n_r[X])\in \bigoplus_{i=1}^{r}\textrm{CH}_{*+u_i-n}(X).
\end{displaymath}
Then $\alpha|_{E_{x_0}}=\Phi_X(\beta)|_{E_{x_0}}$, i.e., $(\alpha-\Phi_X(\beta))|_{E_{x_0}}=0$.
By Lemma \ref{point-open}, $(\alpha-\Phi_X(\beta))|_{E_U}=0$ for a nonempty open subvariety $U$ of $X$.
Set $Z=X-U$.
Notice that $i_{E_Z*}(\pi|_{E_Z})^*=\pi^*i_{Z*}$ on $\textrm{CH}_{*}(Z)$ (\cite[Eexample 17.4.1 (a)]{F}) and $i_{E_U}^*(\xi\cap\eta)=i_{E_U}^*(\xi)\cap i_{E_U}^*(\eta)$ for $\xi\in A^*(E)$, $\eta\in \textrm{CH}_*(E)$.
We have the commutative diagram of localization sequences
\begin{displaymath}
\small{\xymatrix{
 \bigoplus_{i=1}^{r}\textrm{CH}_{*+u_i-n}(Z)\ar[d]^{\Phi_Z} \ar[r]^{i_{Z*}}&
 \bigoplus_{i=1}^{r}\textrm{CH}_{*+u_i-n}(X)\ar[d]^{\Phi_X} \ar[r]^{i_U^*}&
 \bigoplus_{i=1}^{r}\textrm{CH}_{*+u_i-n}(U) \ar[d]^{\Phi_U}\ar[r]& 0\\
 \textrm{CH}_{*}(E_Z)    \ar[r]^{i_{E_Z*}}&
 \textrm{CH}_{*}(E)  \ar[r]^{i_{E_U}^*\quad} &
 \textrm{CH}_{*}(E_U)   \ar[r]& 0. }}
\end{displaymath}
So $\alpha-\Phi_X(\beta)=i_{E_Z*}(\zeta)$ for some $\zeta\in \textrm{CH}_{*}(E_Z)$.
Since  $\textrm{dim}_k Z<\textrm{dim}_k X$,  $\Phi_Z$ is surjective by induction.
There exists $\omega\in \bigoplus_{i=1}^{r}\textrm{CH}_{*+u_i-n}(Z)$ such that $\zeta=\Phi_Z(\omega)$.
Hence $\alpha=\Phi_X(\beta)+i_{E_Z*}(\Phi_Z(\omega))=\Phi_X(\beta+i_{Z*}(\omega))$.
We proved that $\Phi_X$ is surjective.
\end{proof}


\begin{cor}\label{ppp}
Let $X$ be a smooth complete variety  satisfying the K\"{u}nneth property \emph{(}resp. $\mathbb{Q}$-K\"{u}nneth property\emph{)}.
Then
\begin{displaymath}
\times:\emph{CH}_*(X)\otimes_{\mathbb{Z}}\emph{CH}_*(Y)\rightarrow \emph{CH}_*(X\times Y)
\end{displaymath}
\emph{(}resp.
\begin{displaymath}
\times:\emph{CH}_*(X)_{\mathbb{Q}}\otimes_{\mathbb{Q}}\emph{CH}_*(Y)_{\mathbb{Q}}\rightarrow \emph{CH}_*(X\times Y)_{\mathbb{Q}})
\end{displaymath}
is an isomorphism for any variety $Y$.
\end{cor}
\begin{proof}
We only prove the integral cases.
By \cite[Theorem 2.1 (ii)]{ES2}, $\textrm{CH}_*(X)$ is a free abelian group.
Let  $\{\,\gamma_i\in\textrm{CH}_{p_i}(X)\,|\,1\leq i\leq r\,\}$ be a homogeneous  $\mathbb{Z}$-basis of  $\textrm{CH}_*(X)$.
Since $X$ is a smooth  variety, there exists $e_i\in A^{\textrm{dim}_kX-p_i}(X)$ such that $e_i\cap [X]=\gamma_i$ for $1\leq i\leq r$.
For any variety $Y$, let $pr_1:X\times Y\rightarrow X$ and $pr_2:X\times Y\rightarrow Y$ be projections from $X\times Y$ onto $X$ and $Y$ respectively.
For any $y\in Y(k)$, $(pr_1^*e_i)|_{X\times \{y\}}$, $\cdots$, $(pr_1^*e_r)|_{X\times \{y\}}$ is a basis of $\textrm{CH}_*(X\times \{y\})$.
We view  $pr_2:X\times Y\rightarrow Y$ as a fiber bundle over $Y$.
By Theorem \ref{L-H},
\begin{displaymath}
\sum_{i=1}^{r}\gamma_i \times \bullet=\sum_{i=1}^{r}pr_1^*e_i\cap pr_2^*(\bullet):\bigoplus_{i=1}^{r}\textrm{CH}_{*-p_i}(Y)\rightarrow \textrm{CH}_{*}(X\times Y)
\end{displaymath}
is an isomorphism, which implies the conclusion.
\end{proof}

\begin{cor}\label{pppp}
Let $X$ and $Y$ be smooth complete varieties  satisfying the K\"{u}nneth property \emph{(}resp.  $\mathbb{Q}$-K\"{u}nneth property\emph{)}.
Then $X\times Y$ satisfies the K\"{u}nneth property \emph{(}resp.  $\mathbb{Q}$-K\"{u}nneth property\emph{)}.
\end{cor}
\begin{proof}
We only prove the integral cases.
For any variety $Z$, we have the commutative diagram
\begin{equation}\label{commutative diagram}
\xymatrix{
\textrm{CH}_{*}(X)\otimes_{\mathbb{Z}}\textrm{CH}_*(Y)\otimes_{\mathbb{Z}}\textrm{CH}_*(Z)\ar[d]_{id\otimes \times} \ar[r]^{\quad\quad\times\otimes id}  & \textrm{CH}_{*}(X\times Y)\otimes_{\mathbb{Z}}\textrm{CH}_*(Z) \ar[d]^{\times} \\
 \textrm{CH}_{*}(X)\otimes_{\mathbb{Z}}\textrm{CH}_*(Y\times Z)\ar[r]^{\quad\quad\times}  & \textrm{CH}_{*}(X\times Y\times Z).}
\end{equation}
Since $X$ and $Y$  satisfy the  K\"{u}nneth property, the two rows and the left vertical map in (\ref{commutative diagram}) are all isomorphisms by Corollary \ref{ppp}.
So the right vertical map is also an isomorphism.
We complete the proof.
\end{proof}

\subsection{Algebraic fiber bundles of Leray-Hirsch type}

\begin{defn}\label{def}
An algebraic fiber bundle is said to be of \emph{Leray-Hirsch type} (resp. \emph{$\mathbb{Q}$-Leray-Hirsch type}), if it satisfies the assumptions in Theorem \ref{L-H}.
\end{defn}
Clearly, if an algebraic fiber bundle is  of Leray-Hirsch type, it is also  of  $\mathbb{Q}$-Leray-Hirsch type.
If the base variety $X$ is \emph{smooth},
that the fiber bundle $E$ is of  Leray-Hirsch type (resp. $\mathbb{Q}$-Leray-Hirsch type) means
that \emph{its general fiber $F$ is a smooth complete variety satisfying the  K\"{u}nneth property \emph{(}resp. $\mathbb{Q}$-K\"{u}nneth property\emph{)} and there exist cycle classes $e_1$, $\ldots$, $e_r$ in $\emph{CH}^*(E)$  \emph{(}resp. $\emph{CH}^*(E)_{\mathbb{Q}}$\emph{)} such that their restrictions $e_1|_{E_x}$, $\ldots$, $e_r|_{E_x}$ freely generate $\emph{CH}^*(E_x)$ \emph{(}resp. $\emph{CH}^*(E_x)_{\mathbb{Q}}$\emph{)}  for every point $x\in X(k)$.}
In such case,
\begin{displaymath}
\sum_{i=1}^{r}\pi^*(\bullet)\cup e_i:\bigoplus_{i=1}^{r}\textrm{CH}^{*-u_i}(X)\rightarrow \textrm{CH}^{*}(E)
\end{displaymath}
\textrm{(}resp.
\begin{displaymath}
\sum_{i=1}^{r}\pi^*(\bullet)\cup e_i:\bigoplus_{i=1}^{r}\textrm{CH}^{*-u_i}(X)_{\mathbb{Q}}\rightarrow \textrm{CH}^{*}(E)_{\mathbb{Q}})
\end{displaymath}
is an isomorphism, where $e_i\in \textrm{CH}^{u_i}(E)$ (resp. $\textrm{CH}^{u_i}(E)_{\mathbb{Q}}$).
Consider the commutative diagram
\begin{equation}\label{commu}
\xymatrix{
\bigoplus_{i=1}^{r}\textrm{CH}^{*-u_i}(X)_{\mathbb{Q}}\ar[d]_{\bigoplus_{i=1}^{r}\textrm{cl}_{X,\mathbb{Q}}^{*-u_i}} \quad\ar[r]^{\quad\quad\sum_{i=1}^{r}\pi^*(\bullet)\cdot e_i}  & \quad\textrm{CH}^{*}(E)_\mathbb{Q} \ar[d]^{\textrm{cl}_{E,\mathbb{Q}}^{*}} \\
 \bigoplus_{i=1}^{r}H^{2*-2u_i}(X,\mathbb{Q})\qquad \ar[r]^{\quad\quad\sum_{i=1}^{r}\pi^*(\bullet)\cup \textrm{cl}^{u_i}_{E,\mathbb{Q}}(e_i)}  & \quad H^{2*}(E,\mathbb{Q}).}
\end{equation}
For any $x\in X(\mathbb{C})$, since $E_x$ satisfies the $\mathbb{Q}$-K\"{u}nneth property,
$\textrm{cl}_{E_x}^*:\textrm{CH}^*(E_x)_{\mathbb{Q}}\rightarrow H^*(E_x,\mathbb{Q})$
is an isomorphism by the following Lemma \ref{chow-homology}, hence $\textrm{cl}^{u_1}_{E,\mathbb{Q}}(e_1)|_{E_x}$, $\ldots$, $\textrm{cl}^{u_r}_{E,\mathbb{Q}}(e_r)|_{E_x}$ is a basis of $H^*(E_x,\mathbb{Q})$.
By the Leray-Hirsch theorem for singular cohomology, the lower row in (\ref{commu}) is an isomorphism.
By Theorem \ref{L-H}, the upper row in (\ref{commu}) is an isomorphism.
We easily obtain an isomorphism
\begin{equation}\label{L-H-alg}
\sum\limits_{i=1}^{r}\pi^*(\bullet)\cup \textrm{cl}^{u_i}_{E,\mathbb{Q}}(e_i): \bigoplus_{i=1}^{r}H_{alg}^{*-2u_i}(X,\mathbb{Q})  \tilde{\rightarrow} H_{alg}^{*}(E,\mathbb{Q}).
\end{equation}

\begin{lem}\label{chow-homology}
Let $X$ be a smooth complex complete variety.
Assume that  class $\Delta\in \emph{CH}_*(X\times X)$ of the diagonal $X\times X$ is in the image of
$\times:\emph{CH}_*(X)_{\mathbb{Q}}\otimes_{\mathbb{Q}}\emph{CH}_*(X)_{\mathbb{Q}}\rightarrow \emph{CH}_*(X\times X)_{\mathbb{Q}}$.
Then the cycle map $\emph{cl}^*_{X,\mathbb{Q}}:\emph{CH}^*(X)_{\mathbb{Q}} \rightarrow H^* (X, \mathbb{Q})$ is an isomorphism.
\end{lem}
\begin{proof}
Suppose that $\Delta=\sum_{i=1}^r c_i\cdot \alpha_i\times \beta_i$, where $\alpha_i$, $\beta_i\in \textrm{CH}_*(X)_{\mathbb{Q}}$ and $c_i\in\mathbb{Q}$ for $1\leq i\leq r$.
For any $\gamma\in \textrm{CH}_*(X)_{\mathbb{Q}}$,
\begin{equation}\label{diagonal}
\gamma=pr_{1*}(\Delta\cdot pr_2^*\gamma)=\sum_{i=1}^r c_i\left(\int_X \beta_i\cdot\gamma\right)\cdot \alpha_i.
\end{equation}

Assume that  $\textrm{cl}^*_{X,\mathbb{Q}}(\gamma)=0$.
Set $\gamma=\sum_{i=1}^l p_i [Z_i]^{alg}$, where $p_1$, $\cdots$, $p_l$ $\in\mathbb{Q}$ and $Z_1$, $\cdots$, $Z_l$ are subvarieties of $X$.
Then there exists $n\in\mathbb{Z}$ such that $np_1$, $\cdots$, $np_l$ $\in\mathbb{Z}$ and $\sum_{i=1}^l np_i Z_i$ is integral homological to zero, hence is numerical to zero.
So
\begin{displaymath}
\int_X \beta_i\cdot\gamma= \frac{1}{n}\int_X \beta_i\cdot\left(\sum_{i=1}^l np_i Z_i\right)=0.
\end{displaymath}
By (\ref{diagonal}), $\gamma=0$.
We proved that $\textrm{cl}^*_{X,\mathbb{Q}}$ is injective.

Denote  by $\tilde{\Delta}$, $\tilde{\alpha}_i$, $\tilde{\beta}_i$ the corresponding class in $H^* (X, \mathbb{Q})$ of $\Delta$, $\alpha_i$, $\beta_i$ respectively.
Then  $\tilde{\Delta}=\sum_{i=1}^r c_i\cdot \tilde{\alpha}_i\times \tilde{\beta}_i$.
Analogue to (\ref{diagonal}),
\begin{displaymath}
u=\sum_{i=1}^r c_i\left(\int_X \tilde{\beta}_i\cup u\right)\cdot \tilde{\alpha}_i=\textrm{cl}^*_{X,\mathbb{Q}}\left(\sum_{i=1}^r c_i\left(\int_X \tilde{\beta}_i\cup u\right)\cdot \alpha_i\right)
\end{displaymath}
for any $u\in H^*(X,\mathbb{Q})$, which implies that $\textrm{cl}^*_{X,\mathbb{Q}}$ is surjective.
\end{proof}

Assume that $X$ is a smooth complex projective variety.  Now, we list some examples of algebraic fiber bundles of Leray-Hirsch type as follows.

$\bullet$ \emph{The trivial fiber bundle $X\times F$ over $X$, where $F$ is a smooth complex complete variety satisfying the  K\"{u}nneth property.}
Obviously, its general fiber is $F$.
Let $pr_1:X\times F\rightarrow X$ and $pr_2:X\times F\rightarrow F$ be projections from $X\times F$ onto $X$ and $F$ respectively.
Assume that  $\{\,f_i\,|\,1\leq i\leq r\,\}$ is a homogeneous  $\mathbb{Z}$-basis of  $\textrm{CH}_*(F)$.
For any $x\in X(k)$, $(pr_2^*f_1)|_{\{x\}\times F}$, $\cdots$, $(pr_2^*f_r)|_{\{x\}\times F}$ is a basis of $\textrm{CH}^*(\{x\}\times F)$.

$\bullet$ \emph{A quadric bundle associated to a vector bundle with a nondegenerate hyperbolic quadric form over  $X$.}
Let $V$ be a vector bundle of rank $N$ ($=2n$ or $2n+1$) with a nondegenerate hyperbolic quadric form over a smooth projective variety $X$  and let $Q\subseteq \mathbb{P}(V)$ be the associated quadric bundle.
The fiber over any closed point of $X$ is a hyperbolic quadric in $\mathbb{C}P^{N-1}$, which has a cellular decomposition.
Set $h=c_1(\mathcal{O}_{\mathbb{P}(V)}(1))|_{Q}\in \textrm{CH}^1(Q)$ and denote by $\gamma\in \textrm{CH}^n(Q)$ the class of a maximal isotropic subbundle.
Then $1$, $h$, $h^2$, \ldots, $h^{n-1}$, $\gamma$, $h\gamma$, \ldots, $h^{n-1}\gamma$ is a basis of $\textrm{CH}^*(Q)$ as a $\textrm{CH}^*(X)$-module (\cite[Theorem 7]{EG1}), whose restrictions  freely generate $\textrm{CH}^*(Q_x)$  for any $x\in X(\mathbb{C})$ (\cite[Lemma 1]{EG1}).

$\bullet$ \emph{A flag bundle  over $X$.}
The general fiber of a flag bundle is a flag variety, which is a cellular variety.
Suppose that $V$ is an algebraic vector bundle with rank $n$ over $X$ and $n_1$, $\ldots$, $n_r$ is a sequence of positive integers with $\sum_{i=1}^rn_i=n$, such that the fiber $E_x$ of $E$ over $x\in X(\mathbb{C})$ is the flag variety $Fl(n_1,\ldots,n_r)(V_x)=$
\begin{displaymath}
\begin{aligned}
\{(V_1,\ldots,V_{r-1})|&0=V_0\subsetneq V_1\subsetneq V_2\subsetneq \ldots \subsetneq V_{r-1}\subsetneq V_r=V_x, \mbox{ where } V_i \mbox{ is a } \\
& \mbox{ complex vector space with dimension } \sum_{j=1}^{i}n_j \mbox{  for } 1\leq i\leq r\}.
\end{aligned}
\end{displaymath}
For $0\leq i\leq r$, let $E_i$ be the  universal subbundle over $E$ whose fiber over the point $(V_1,\ldots,V_{r-1})$ is $V_i$. Notice that $E_0=0$ and $E_r=\pi^*V$, where $\pi$ is the projection form $F$ onto $X$. Denote by $E^{(i)}=E_i/E_{i-1}$ the successive  universal quotient bundles  and by $c_j^{(i)}=c_j(E^{(i)})\in \textrm{CH}^*(E)$ the $j$-th Chern class of $E^{(i)}$ for $1\leq i\leq r$. As we know, $\textrm{CH}^*(E)$ is a  $\textrm{CH}^*(X)$-module via $\pi^*$ freely generated by the monomials $e_1$, $\ldots$, $e_k$ on $c_1^{(1)}$, $\ldots$, $c_{n_1}^{(1)}$, $c_1^{(2)}$, $\ldots$, $c_{n_2}^{(2)}$, $\ldots$, $c_1^{(r)}$, $\ldots$, $c_{n_r}^{(r)}$. For every $x\in X$, the restrictions $c_{j_i}^{(i)}|_{E_x}$ ($1\leq i\leq r$, $1\leq j_i\leq n_i$) to $E_x$ are exactly the Chern classes of the successive  universal quotient bundles of the flag variety $E_x$. So $e_1|_{E_x}$, $\ldots$, $e_k|_{E_x}$ freely generate the Chow ring $\textrm{CH}^*(E_x)$ as an abelian group.

\section{Hodge and Standard Conjectures for Fiber bundles}



\begin{lem}\label{surj-lem-C}
Let $f:X\rightarrow Y$ be a surjective morphism of smooth complex projective varieties. Then $C(X)$ implies $C(Y)$.
\end{lem}
\begin{proof}
Set $r=\textrm{dim}_{\mathbb{C}}X-\textrm{dim}_{\mathbb{C}}Y$.
There exists a subvariety $V$ of $X$ of codimension $r$ such that $f|_V:V\rightarrow Y$ is a surjective morphism of projective varieties with the same dimensions. Suppose that $f|_V$ is of degree $d$. Then $f_*[V]^{top}=d\cdot[Y]^{top}$.
By the projection formula,
\begin{equation}\label{proformula}
(f\times f)_*\left([V\times V]^{top}\cup (f\times f)^*(\alpha)\right)=d^2\cdot\alpha
\end{equation}
for any $\alpha\in H_{alg}^{*}(Y\times Y,\mathbb{Q})$.
Let $Z$ be any algebraic cycle of codimension $k$ with rational coefficients on $Y\times Y$.
Denote by $S_p\in H^p(X,\mathbb{Q})\otimes_{\mathbb{Q}} H^{2k-p}(X,\mathbb{Q})$ the K\"{u}nneth components of $(f\times f)^*([Z]^{top})\in H_{alg}^{*}(X\times X,\mathbb{Q})$.
Since $C(X)$, $S_p$ for all $p$ are algebraic by \cite[Proposition 2.6]{K}.
So the K\"{u}nneth components of $[Z]^{top}$ in  $H^p(Y,\mathbb{Q})\otimes_{\mathbb{Q}} H^{2k-p}(Y,\mathbb{Q})$
is $\frac{1}{d^2}\cdot(f\times f)_*\left([V\times V]^{top}\cup S_p\right)$, which is obviously algebraic.
By \cite[Proposition 2.6]{K}, $C(Y)$.
\end{proof}

Now, we provide a proof of Theorem \ref{L-H2}.
\begin{proof}
Let $\pi:E\rightarrow X$ be the algebraic fiber bundle and let $F$ be its general fiber.
Set $n=\textrm{dim}_{\mathbb{C}}X$ and $m=\textrm{dim}_{\mathbb{C}}F$.

$(1)$
Suppose that $V_1$, $\ldots$, $V_r$ are algebraic cycles with integral coefficients on $E$  such that their restrictions $[V_1]^{alg}|_{E_x}$, $\ldots$, $[V_r]^{alg}|_{E_x}$ freely generate $\textrm{CH}^*(E_x)$  for every point $x\in X(\mathbb{C})$.
Set $\textrm{codim}_{\mathbb{C}}V_i=u_i$ for $1\leq i\leq r$.
Consider the commutative diagram
\begin{displaymath}
\xymatrix{
\bigoplus_{i=1}^{r}\textrm{CH}^{*-u_i}(X)\ar[d]_{\bigoplus_{i=1}^{r}\lambda^{*-u_i}_X} \quad\ar[r]^{\quad\quad\sum_{i=1}^{r}\pi^*(\bullet)\cdot [V_i]^{alg}}  & \quad\textrm{CH}^{*}(E) \ar[d]^{\lambda^*_{E}} \\
 \bigoplus_{i=1}^{r}Hdg^{*-u_i}(X,\mathbb{Z})\qquad \ar[r]^{\quad\quad\sum_{i=1}^{r}\pi^*(\bullet)\cup [V_i]^{top}}  & \quad Hdg^*(E,\mathbb{Z}).}
\end{displaymath}
For any $x\in X(\mathbb{C})$, since $E_x$ satisfies the K\"{u}nneth property,  $\textrm{cl}_{E_x}^*:\textrm{CH}^*(E_x)\rightarrow H^*(E_x,\mathbb{Z})$ is an isomorphism by \cite[Theorem 2.1 (iii)]{ES2}, hence $[V_1]^{top}|_{E_x}$, $\ldots$, $[V_r]^{top}|_{E_x}$ freely generate $H^*(E_x,\mathbb{Z})$.
The lower row is an isomorphism by the Hodge decompositions and the Leray-Hirsch theorem for singular cohomology.
By Theorem \ref{L-H}, the upper row is an isomorphism.
Hence $\lambda^*_X$ is surjective if and only if  $\lambda^*_E$ is surjective.

$(2)$  Let  $e_1$, $\ldots$, $e_r\in \textrm{CH}^*(E)_{\mathbb{Q}}$ be of pure degrees $u_1$ , $\ldots$, $u_r$ respectively  such that their restrictions $e_1|_{E_x}$, $\ldots$, $e_r|_{E_x}$ is a basis of  $\textrm{CH}^*(E_x)_{\mathbb{Q}}$  for every point $x\in X(\mathbb{C})$.

$A(E)\Rightarrow A(X)$: See \cite[Theorem 1.5]{Tan}.

$A(X)\Rightarrow A(E)$:
Set  $h^p(\bullet)=\textrm{dim}_{\mathbb{Q}}H_{alg}^{2p}(\bullet,\mathbb{Q})$.
By (\ref{L-H-alg}),  $h^p(E)=\sum_{i=0}^rh^{p-u_i}(X)$ for any $p\in \mathbb{N}$.
By  Lemma \ref{chow-homology}, $\textrm{dim}_{\mathbb{Q}}\textrm{CH}^j(E_x)_{\mathbb{Q}}=b_{2j}(E_x)=b_{2j}(F)$, where $b_i(\bullet)$ is the $i$-th Betti number.
The cardinal of $\{\,u_i\,|\,u_i=j,\,1\leq i\leq r\,\}$ is $b_{2j}(F)$,  since $e_1|_{E_x}$, $\ldots$, $e_r|_{E_x}$ is a basis of $\textrm{CH}^*(E_x)_{\mathbb{Q}}$.
Hence
\begin{equation}\label{Kunnthe-alg}
h^p(E)=\sum_{j=0}^n\sum\limits_{\substack{1\leq i\leq r\\u_i=j}}h^{p-u_i}(X)=\sum_{j=0}^nb_{2j}(F)h^{p-j}(X)
\end{equation}
for any $p\in \mathbb{N}$.
Since $A(X)$, i.e., $h^j(X)=h^{n-j}(X)$ for all $j\in \mathbb{N}$.
By the Poincar\'{e} duality theorem,  $b_{2j}(F)=b_{2m-2j}(F)$ for all $j$.
By (\ref{Kunnthe-alg}), we easily get $h^{m+n-p}(E)=h^p(E)$.
So $A(E)$.

$A(\bullet)\Leftrightarrow D(\bullet)$: See \cite[Theorem 1]{L}.

$B(E)\Rightarrow B(X)$: See \cite[Theorem 1.6]{Tan}.

$B(X)\Rightarrow B(E)$:
Assume that $B(X)$.
By \cite[p. 603 (1.11)]{Tan}, $A(X\times X)$.
Clearly, $\pi\times id: E\times X\rightarrow X\times X$ is an algebraic fiber bundle  with the general fiber $F$.
For every point $(x,y)\in (X\times X)(\mathbb{C})=X(\mathbb{C})\times X(\mathbb{C})$,
$(e_1\times [X]^{alg})|_{E_x\times \{y\}}$, $\ldots$, $(e_r\times [X]^{alg})|_{E_x\times \{y\}}$ is a basis of $\textrm{CH}^*(E_x\times \{y\})_{\mathbb{Q}}$.
So $A(E\times X)$.
Similarly, $id\times \pi: E\times E\rightarrow E\times X$ is an algebraic fiber bundle  with the general fiber $F$.
For every point $(v,y)\in (E\times X)(\mathbb{C})=E(\mathbb{C})\times X(\mathbb{C})$,
$([E]^{alg}\times e_1)|_{\{v\}\times E_y}$, $\ldots$, $([E]^{alg}\times e_r)|_{\{v\}\times E_y}$ is a basis of $\textrm{CH}^*(\{v\}\times E_y)_{\mathbb{Q}}$.
Hence $A(E\times E)$.
By \cite[p. 603 (1.11)]{Tan}, $B(E)$.

$C(E)\Rightarrow C(X)$: See Lemma \ref{surj-lem-C}.

$C(X)\Rightarrow C(E)$: For any $(x, y)\in (X\times X)(\mathbb{C})=X(\mathbb{C})\times X(\mathbb{C})$, $E_x$ and $E_y$ satisfy the  $\mathbb{Q}$-K\"{u}nneth property.
By Corollary \ref{ppp}, $\times:\textrm{CH}^*(E_x)_{\mathbb{Q}}\otimes_{\mathbb{Q}} \textrm{CH}^*(E_y)_{\mathbb{Q}}\rightarrow\textrm{CH}^*(E_x\times E_y)_{\mathbb{Q}}$ is an isomorphism.
So $\{\,(e_i\times e_j)|_{E_x\times E_y}\,|\,1\leq i\leq r,\,1\leq j\leq r\,\}$ is a basis of $\textrm{CH}^*((E\times E)_{(x,y)})_{\mathbb{Q}}=\textrm{CH}^*(E_x\times E_y)_{\mathbb{Q}}$.
By Corollary \ref{pppp}, $E_x\times E_y$ satisfy the  $\mathbb{Q}$-K\"{u}nneth property.
By (\ref{L-H-alg}),
\begin{displaymath}
\sum_{i,j=1}^{r}(\pi\times\pi)^*(\bullet)\cup(e_i\times e_j):\bigoplus_{i,j=1}^{r}H_{alg}^{*-2u_i-2u_j}(X\times X,\mathbb{Q})\rightarrow H_{alg}^{*}(E\times E,\mathbb{Q})
\end{displaymath}
is an isomorphism.
Let $Z$ be any  algebraic cycle of codimension $k$  with rational coefficients on $E\times E$.
There exist $\alpha_{ij}\in H_{alg}^{2k-2u_i-2u_j}(X\times X,\mathbb{Q})$ such that $[Z]^{top}=\sum_{i,j=1}^{r}(\pi\times\pi)^*(\alpha_{ij})\cup (e_i\times e_j)$.
Denote by $[Z]_p\in H^p(E,\mathbb{Q})\otimes_{\mathbb{Q}} H^{2k-p}(E,\mathbb{Q})$ and  $\alpha^p_{ij}\in H^p(X,\mathbb{Q})\otimes_{\mathbb{Q}} H^{2k-2u_i-2u_j-p}(X,\mathbb{Q})$ the K\"{u}nneth components of $[Z]^{top}$ and  $\alpha_{ij}$ respectively.
Then $[Z]_p=\sum_{i,j=1}^{r}(\pi\times\pi)^*(\alpha^{p-2u_i}_{ij})\cup (e_i\times e_j)$.
Since $C(X)$, $\alpha^p_{ij}$ are algebraic for all $i$, $j$, $p$ by \cite[Proposition 2.6]{K}, so are $[Z]_p$ for all $p$.
By \cite[Proposition 2.6]{K}, $C(E)$.

$\textrm{Hodge}(X,\mathbb{Q})\Leftrightarrow \textrm{Hodge}(E,\mathbb{Q})$: As $(1)$, we can give a similar proof by Lemma \ref{chow-homology} instead of \cite[Theorem 2.1 (iii)]{ES2}.

\end{proof}

\section{Hodge conjecture for product varieties}
The following Lemma \ref{surj-lem} may be well known for experts. We do not find the reference and hence  give a simple proof here.
\begin{lem}\label{surj-lem}
Let $f:X\rightarrow Y$ be a surjective morphism of smooth complex projective varieties. Then $\emph{Hodge}(X,\mathbb{Q})$ implies $\emph{Hodge}(Y,\mathbb{Q})$. Moreover, if there exists an algebraic cycle $V$ with integral coefficients on $X$ with pure codimension  $r=\emph{dim}_{\mathbb{C}}X-\emph{dim}_{\mathbb{C}}Y$ such that $f_*[V]^{top}=[Y]^{top}$ in $H^0(Y,\mathbb{Z})=\mathbb{Z}[Y]^{top}$, then $\emph{Hodge}(X,\mathbb{Z})$ implies $\emph{Hodge}(Y,\mathbb{Z})$.
\end{lem}
\begin{proof}

Assume that $\textrm{Hodge}(X,\mathbb{Q})$. Let $V$, $d$ be the ones in the proof of Lemma \ref{surj-lem-C}.
For any $\alpha\in Hdg^{p}(Y,\mathbb{Q})$, $f^*\alpha=[Z]^{top}$ for an algebraic cycle $Z$ with $\mathbb{Q}$-coefficients on $X$, since $\textrm{Hodge}(X,\mathbb{Q})$.  By the projection formula,
\begin{equation}\label{proformula}
\alpha=1/d\cdot f_*([V]^{top}\cup f^*\alpha)=1/d\cdot \textrm{cl}_Y^{p+r}(f_*([V]^{alg}\cdot [Z]^{alg}))
\end{equation}
is algebraic. So $\textrm{Hodge}(Y,\mathbb{Q})$. The first part  was proved. These arguments are also valid for proving the second part, where we use $\mathbb{Z}$ instead of $\mathbb{Q}$ and notice that $d=1$.
\end{proof}

\begin{lem}\label{L-H1}
Let $X$, $E$ be smooth complex projective varieties and let $\pi:E\rightarrow X$ be an algebraic fiber bundle.
Suppose that there exist algebraic classes $e_1$, $\ldots$, $e_r$ in $H^*(E,\mathbb{Z})$ \emph{(}resp. $H^*(E,\mathbb{Q})$\emph{)}  such that their restrictions $e_1|_{E_x}$, $\ldots$, $e_r|_{E_x}$ freely generate $H^*(E_x,\mathbb{Z})$ \emph{(}resp. $H^*(E_x,\mathbb{Q})$\emph{)} for every $x\in X(\mathbb{C})$. Then $\emph{Hodge}(X,\mathbb{Z})$ implies $\emph{Hodge}(E,\mathbb{Z})$ \emph{(}resp. $\emph{Hodge}(E,\mathbb{Q})$ if and only if $\emph{Hodge}(X,\mathbb{Q})$\emph{)}.
\end{lem}
\begin{proof}
Assume that  $e_i$ is of pure degree $2u_i$ for $1\leq i\leq r$. By the Hodge decompositions and the Leray-Hirsch theorem for singular cohomology, we have an isomorphism
\begin{displaymath}
\sum_{i=1}^{r}\pi^*(\bullet)\cup e_i:\bigoplus_{i=1}^{r}Hdg^{*-u_i}(X,\mathbb{Z})\tilde{\rightarrow} Hdg^*(E,\mathbb{Z}).
\end{displaymath}
Evidently, $\textrm{Hodge}(X,\mathbb{Z})$ implies $\textrm{Hodge}(E,\mathbb{Z})$.
The argument also apply to the case of the Hodge conjeture.
By Lemma \ref{surj-lem}, we complete the proof.
\end{proof}

Now, we give a proof of Theorem \ref{product}.
\begin{proof}

Let $pr_1:X\times Y\rightarrow X$ and $pr_2:X\times Y\rightarrow Y$ be projections from $X\times Y$ onto $X$ and $Y$ respectively.

$(1)$ Fixed a point $y_0\in Y(\mathbb{C})$, $pr_{1*}([X\times \{y_0\}]^{top})=[X]^{top}$ in $H^0(X,\mathbb{Z})$. By Lemma \ref{surj-lem}, we get the part of ``only if".
View $X\times Y$ as a trivial algebraic fiber bundle with general fiber $Y$ over $X$. Let $H^*(Y,\mathbb{Z})$ be freely generated by algebraic classes $\alpha_1$, $\ldots$, $\alpha_r$  as an abelian group.
Then $e_1|_{\{x\}\times Y}$, $\ldots$, $e_r|_{\{x\}\times Y}$ freely generate $H^*(\{x\}\times Y,\mathbb{Z})$ for any $x\in X(\mathbb{C})$,  where $e_i=pr_2^*\alpha_i$ are all algebraic classes for $1\leq i\leq r$.  By Lemma \ref{L-H1},  the part of ``if" holds.

$(2)$ By Lemma \ref{surj-lem}, the part of ``only if" holds. For the other part, we first prove a claim:
The exterior product $\bullet\times\bullet:=\emph{pr}_1^*(\bullet)\cup\emph{pr}_2^*(\bullet)$ gives an isomorphism
\begin{equation}\label{Kunneth-Hdg}
\bigoplus_{p+q=k}Hdg^p(X,\mathbb{Q})\otimes_{\mathbb{Q}}Hdg^q(Y,\mathbb{Q})\rightarrow Hdg^k(X\times Y,\mathbb{Q})
\end{equation}
for any $k$.

Obviously, the map (\ref{Kunneth-Hdg}) is defined well and hence injective by the K\"{u}nneth formula of cohomology with $\mathbb{Q}$-coefficients. We just need to prove that it is surjective. Let $\alpha$ be any class in $Hdg^{k}(X\times Y,\mathbb{Q})$. By the Hodge decomposition, $H^{2q+1}(Y,\mathbb{C})=0$ and $H^{2q}(Y,\mathbb{C})=H^{q,q}(Y)$. Then $H^{2q+1}(Y,\mathbb{Q})=0$ and $H^{2q}(Y,\mathbb{Q})=Hdg^{q}(Y,\mathbb{Q})$. Let $\sigma_1^{q,q}$, $\ldots$, $\sigma_{t_q}^{q,q}$ be a basis of $Hdg^{q}(Y,\mathbb{Q})$.
There exist $\tau_1^{2p}$, $\ldots$, $\tau_{t_q}^{2q}\in H^{2p}(X,\mathbb{Q})$ for $0\leq q\leq k$ such that
\begin{equation}\label{rep1}
\alpha=\sum_{p+q=k}\sum_{i=0}^{t_q}\tau_i^{2p}\times\sigma_i^{q,q}
\end{equation}
by the K\"{u}nneth formula of cohomology with $\mathbb{Q}$-coefficients. Set  $\tau_i^{2p}=\sum_{r+s=2p} \tau_i^{r,s}$, where $\tau_i^{r,s}\in H^{r,s}(X)$. 
Comparing the bidegree in (\ref{rep1}),
\begin{equation}\label{rep3}
\alpha=\sum_{p+q=k}\sum_{i=0}^{t_q} \tau_i^{p,p}\times\sigma_i^{q,q} \mbox{ in $H^{k,k}(X\times Y)$}
\end{equation}
and
\begin{equation}\label{rep4}
\sum\limits_{\substack{r+q=u\\s+q=v}}\sum_{i=0}^{t_q} \tau_i^{r,s}\times\sigma_i^{q,q}=0 \mbox{ in $H^{u,v}(X\times Y)$ for $u\neq v$},
\end{equation}
where $u+v=2k$.  By the K\"{u}nneth formula of Dolbeault cohomology,
\begin{displaymath}
\sum_{i=0}^{t_q} \tau_i^{r,s}\otimes\sigma_i^{q,q}=0 \mbox{ in $H^{r,s}(X)\otimes_{\mathbb{C}} H^{q,q}(Y)$ for $r\neq s$}
\end{displaymath}
from (\ref{rep4}), where $r+s+2q=2k$. So $\tau_i^{r,s}=0$ for $r\neq s$. Then $\tau_i^{2p}=\tau_i^{p,p}\in H^{2p}(X,\mathbb{Q})\cap H^{p,p}(X)=Hdg^p(X,\mathbb{Q})$. By (\ref{rep3}), $\alpha$ is in the image of the map (\ref{Kunneth-Hdg}). We proved the claim.

Consider the commutative diagram
\begin{displaymath}
\xymatrix{
\bigoplus_{p+q=k}\textrm{CH}^{p}(X)_{\mathbb{Q}}\otimes_{\mathbb{Q}}\textrm{CH}^{q}(Y)_{\mathbb{Q}}\ar[d]_{\bigoplus_{p+q=k}\lambda^p_{X,\mathbb{Q}}\otimes\lambda^q_{Y,\mathbb{Q}}} \ar[r]^{\qquad\quad\times}  & \textrm{CH}^{k}(X\times Y)_{\mathbb{Q}} \ar[d]^{\lambda^k_{X\times Y,\mathbb{Q}}} \\
 \bigoplus_{p+q=k}Hdg^{p}(X,\mathbb{Q})\otimes_{\mathbb{Q}}Hdg^{q}(Y,\mathbb{Q}) \ar[r]^{\qquad\quad\times}  & Hdg^{k}(X\times Y,\mathbb{Q}).}
\end{displaymath}
Suppose that $\textrm{Hodge}(X,\mathbb{Q})$ and $\textrm{Hodge}(Y,\mathbb{Q})$. Then the first vertical map is surjective. The isomorphism (\ref{Kunneth-Hdg}) implies that $\lambda^k_{X\times Y,\mathbb{Q}}$ is surjective. We complete the proof.
\end{proof}


\end{document}